\documentclass[%
 reprint,
 amsmath,amssymb,
 aps,
]
{revtex4-1}

\usepackage{graphicx}
\usepackage{dcolumn}
\usepackage{bm}

\usepackage{hyperref}
\usepackage{mathtools}
\mathtoolsset{showonlyrefs=true}
\usepackage{amsmath}
\usepackage{amssymb}
\usepackage{amsthm}
\newtheorem{lemma}{Lemma}
\newtheorem{battle}{Battle}
\newtheorem{theorem}{Theorem}
\newtheorem{remark}{Remark}
\newtheorem{corollary}{Corollary}

\begin{document}


\title{The Battle of Infinity: Explosive Demand Surge vs Gigantic Service Providers}

\author{Hiroshi Toyoizumi}
 \altaffiliation[Also at ]{Dept. of Applied Mathematics, Waseda University.}
\affiliation{%
Graduate School of Accountancy, Waseda University
}%

%


\date{\today}

\begin{abstract}
Infinite server queues have ultimate processing power to accommodate explosive demand surges. We provide a new stability criterion based on the Borel-Cantelli lemma to judge whether the infinite server safely accommodates heavy-tailed demands. We illustrate the battles between heavy-tailed demand and infinite servers in detail. In particular, we show some cases where the explosive demand overwhelms the infinite server queue. The medical demand caused by pandemics such as the COVID-19 creates huge stress to the healthcare system. This framework indicates that healthcare systems need to account for the tail behavior of the cluster size and hospital stay length distributions to check the stability of their systems during pandemics.

==========

[Popular Summary]

Healthcare systems are under pressure due to the COVID-19 pandemic. Social gatherings may create clusters of infections, and the resulting patient stream causes a shortage of beds, medical supplies, and medical staff. In order to meet the explosive surge of medical demand, healthcare systems are reinforced, sometimes even by building temporary hospitals overnight.

The infinite server queue is the ultimate model of such an idealized hospital that serves any number of patients without delay. We find new criteria for the stability of the infinite server queue: the balance of the tail behavior of the cluster size and hospital stay length. Even though each cluster brings a finite number of patients, if the stochastic fluctuation of the cluster and the hospital stay is a volatile power-law tail distribution, the idealized hospital may be overwhelmed by the medical demand and collapse.   

Our results suggest that healthcare systems should take into account the balance of the tail behavior of the cluster size and hospital stay length. Moreover, in some extreme cases, reinforcement of healthcare systems is not the best strategy. Instead, reducing the medical demand by controlling social gatherings and lengths of hospital stays is crucial.

\end{abstract}

\maketitle


\section{Introduction}
In pandemics such as COVID-19, super-spreading events (or clusters) associated with large social gatherings create sudden explosions of infections, and the resulting stream of patients becomes a huge stress to the healthcare system \cite{On-Kwok:2020bs,Iritani:2020xr}. Infinite server queues that can provide immediate services to any number of patients simultaneously are the ultimate models of idealized healthcare systems.  

For another example, in the era of social media, cascaded demand buildup to the most popular service in a short period of time is also common in the winner-takes-all-type competitive market \cite{easley2010networks,levis2009winners,porter2001strategy,prakash2012winner}. Friend networks on social media are modelled by scale-free networks, and the surges in demand for popular services have power-law distributions. Infinite server queues may be used to analyze such service providers accommodating huge surges in demand.

To analyze the stability of systems, we employ a method called queuing theory, which has been developed to evaluate complex service systems \cite{baccelli2013elements,wolff1989stochastic,kleinrock1975queueing}.  In particular, we use an infinite server queue with the batch arrival of customers for modelling the idealized system accommodating the explosive demand sureges. Even though the term may suggest otherwise, infinite server queues do not make customers wait in queues, because they have an infinite number of servers. For example, an infinite server queue can be an idealized hospital that provides immediate treatment to any number of patients arriving simultaneously, or it can be a gigantic platformer that provides immediate service to cascaded demand fuelled by social media. We show that these ideal hospitals or gigantic platformers may be overtaken by demand surge, and thus, would explode.  This shows that, there may be some extreme cases: we need to focus on controlling the demand surges by suppressing the size of social gatherings or reducing hospital stay length.

\section{Discrete fractional power law distribution}
In order to model demand surges, we define a family of power law distributions in the form of a generalized geometric distribution parameterized with $p>0$ as
\begin{align}\label{eq: Power Law Type}
P(X=k) = \left(1-\frac{k}{p+k}\right) \prod_{i=1}^{k-1}\frac{i}{p+i}. 
\end{align}
for $k=1,2, \dots $. Note that this  discrete $p$-th order fractional power law distribution is a generalized geometric distribution with a non-homogenous parameter $i/(p+i)$ \cite{Mandelbaum2007}, which has an asymptotic tail probability of the order of $O(k^{-p})$, and the moment $E[X^{q}] = \infty$ for $q\geq p$ (see Figure \ref{fig:LoglogPlot.pdf}).  
\begin{figure}[tbp]
\begin{center}
\includegraphics[width=0.5\textwidth]{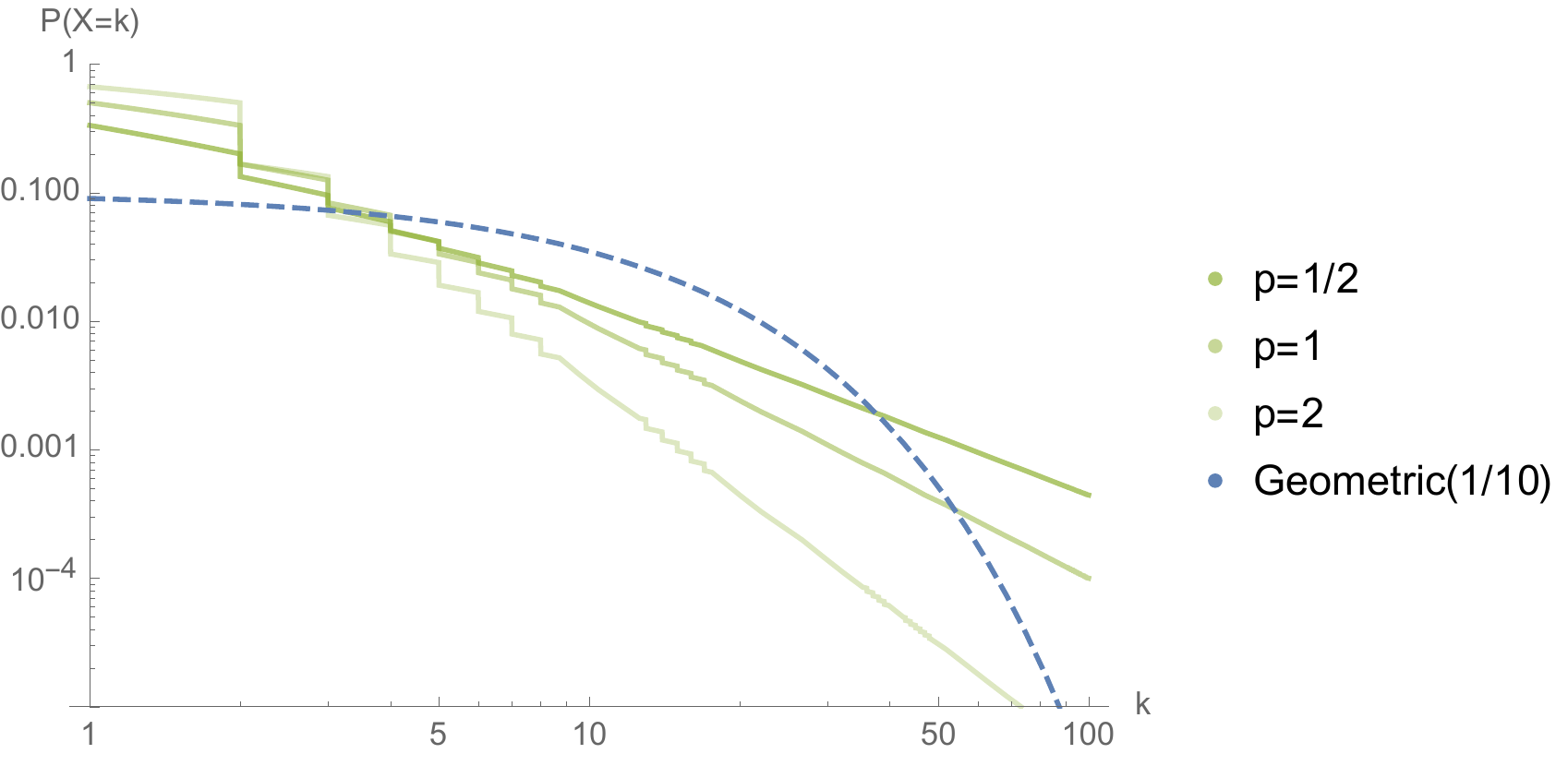}
\caption{Log-log plots of the $p$-th order discrete fractional power law distribution \eqref{eq: Power Law Type} for $p=1/2, 1,2$ (the tail gets heavier for smaller values of $p$), which are compared with the geometric distribution with the parameter $1/10$.}
\label{fig:LoglogPlot.pdf}
\end{center}
\end{figure}
When $p$ is a positive integer, \eqref{eq: Power Law Type} is reduced to 
\begin{align}\label{eq: Power Law Type Integer}
P(X=k) = \frac{p \cdot p!}{k(k+1) \cdots (k+p)}, 
\end{align}
It is known that preferentially attached (scale-free) networks have degree distributions given by \eqref{eq: Power Law Type Integer} with $p=2$, and they have the infinite second moment \cite{dorogovtsev2000structure,durrett2007random}. Friend networks are considered preferentially attached networks, and may serve as suitable models for the size distribution of social gatherings.

\section{Infinite server queue and its stability}\label{section:Infinite server queue and its stability}
We consider the infinite server queue with batch arrivals. Let $\dots < T_{-1}< T_{0} \leq 0 < T_{1} < T_{2}<\dots$ be the sequence of the batch arrival times with finite intensity $\lambda$. Each customer in a batch of size $X_{n}$ arriving at $T_{n}$ is served immediately, and stays in the system for $S_{n,1}, S_{n,2}, \dots , S_{n,X_{n}}$, and then leaves the system. We allow a large stochastic fluctuation of the batch size such that $E[X_{n}] = \infty$, but assume that batch sizes and sojourn times are finite, i.e., $P(X_{n} < \infty)=1$ and $P(S_{n,i} < \infty)=1$.  

We evaluate $L$, the number of customers existing in the system at time $0$, given that the system started sufficiently long before. In this setting, by adding the older batches one by one from time $0$, $L$ is regarded as its monotone-increasing limit. Although batch sizes are finite and each customer spends a finite time in the system, $L$ can be infinite, causing the system to explode. We aim to obtain the condition under which the system is stable with the condition $P(L<\infty) = 1$.

Let $A_{-k,i}= \{S_{-k,i}>-T_{-k}\}$ be events where the customer $i$ in the batch arriving at $T_{-k}$, and is still in the system at time $0$. Then, by summing up all the values, we have
\begin{align}
L = \sum_{k=0}^{\infty}\sum_{i=1}^{X_{-k}}1_{A_{-k,i}}.
\end{align}
By applying the Borel-Cantelli lemma to possibly correlated events $B_{-k}= \cup_{i=1}^{X_{^{k}}}A_{-k,i}$, we can prove that the system is stable when $P(L<\infty) = 1$, if the batch arrival rate $\lambda$ and $E[\max_{1\leq i \leq X_{k}}S_{i}]$ are finite (see details in Appendix \ref{Appendix: Stability Criteria}).

Thus, balancing the sojourn time $S$ (hospitalization time) and the batch size $X$ (cluster size) is necessary to achieve stability. Specifically, using the Holder and Jensen inequalities, we extend the idea presented in \cite{Devroye:1979fk} to obtain the bound of $E[\max_{1\leq i \leq X}S_{i}]$ as
\begin{align}\label{ineq: Sx}
E\left[\max_{1\leq i \leq X}S_{i}\right] \leq \{ E[S^{p}]\}^{1/p} E[X^{1/p}]. 
\end{align}
when $X$ and $S$ are independent and $E[X^{1/p}] < \infty$ and $E[S^{p}] < \infty$ for some $p>0$ (see details in Appendix \ref{Appendix:The Balance Criteria of Stability}).

\section{Battles against explosive demand surges}\label{section:Battles against explosive demand surges}
In the following, we further assume that the batch arrival is a Poisson process with rate $\lambda$, which corresponds to assuming that the social gathering events are independently and randomly organized. In this case, we can explicitly derive the probability generating function of $L$ as
\begin{align}\label{eq:generating function of L}
E\left[z^{L}\right]=\exp \left[ - \lambda \left\{ \int_{0}^{\infty}\left(1- E\left[z^{M(s)}\right]\right)ds \right\}\right]. 
\end{align}
if it exists, where $M(t)=\sum_{i=1}^{X}1_{\{S_{i}>t\}}$ (see details of the derivations in Appendix \ref{Appendix:Infinite Server Queues with Poisson Batch Arrival}). 

We analyze the battle of infinite server queue against the explosive demand surge in detail using the balance criteria \eqref{ineq: Sx}, and the discrete fractional power law distribution defined by \eqref{eq: Power Law Type}, where \eqref{eq:generating function of L} can be explicitly calculated via \eqref{eq: stationary pgf for independent sojourn time} in Appendix \ref{Appendix:Infinite Server Queues with Poisson Batch Arrival}.

\begin{battle}[Light-tailed sojourn time; the infinite server always wins]\label{example: Power-law batch size and exponential sojourn time; the infinite server wins}
We assume that the sojourn times are independent and exponentially distributed (set $E[S] =1/\mu$), and the batch size $X$ is a discrete fractional power law distribution with $p=1$. Even though $E[X]$ is infinite, the tail of the sojourn time is sufficiently light to have finite $E[\max_{1\leq i \leq X}S_{i}] \approx E[\log X]$, and then, the queue is stable (see Corollary \ref{corollary:log bound} in Appendix \ref{Appendix:Infinite Server Queues with Poisson Batch Arrival} and \cite{Cong:1994fk,Yajima:2017fv}). In this case, the probability generating function of the stationary distribution $L$ is given by
\begin{align}
E\left[z^{L}\right]=\exp \left[ - \rho \sum_{n=1}^{\infty}\frac{1-z^{n}}{n^{2}} \right]. 
\end{align}
where $\rho = \lambda/\mu$ (see details in Appendix \ref{Appendix:Infinite Server Queues with Poisson Batch Arrival}). 
Note that $L$ itself is heavy-tailed, and $E[L] = \infty$). Figure \ref{fig:distribution.pdf} shows examples of the probability distribution $P(L=n)$, obtained by checking the expansion with respect to $z$ around $0$.
\end{battle}

\begin{figure}[tbp]
\begin{center}
\includegraphics[width=0.5\textwidth]{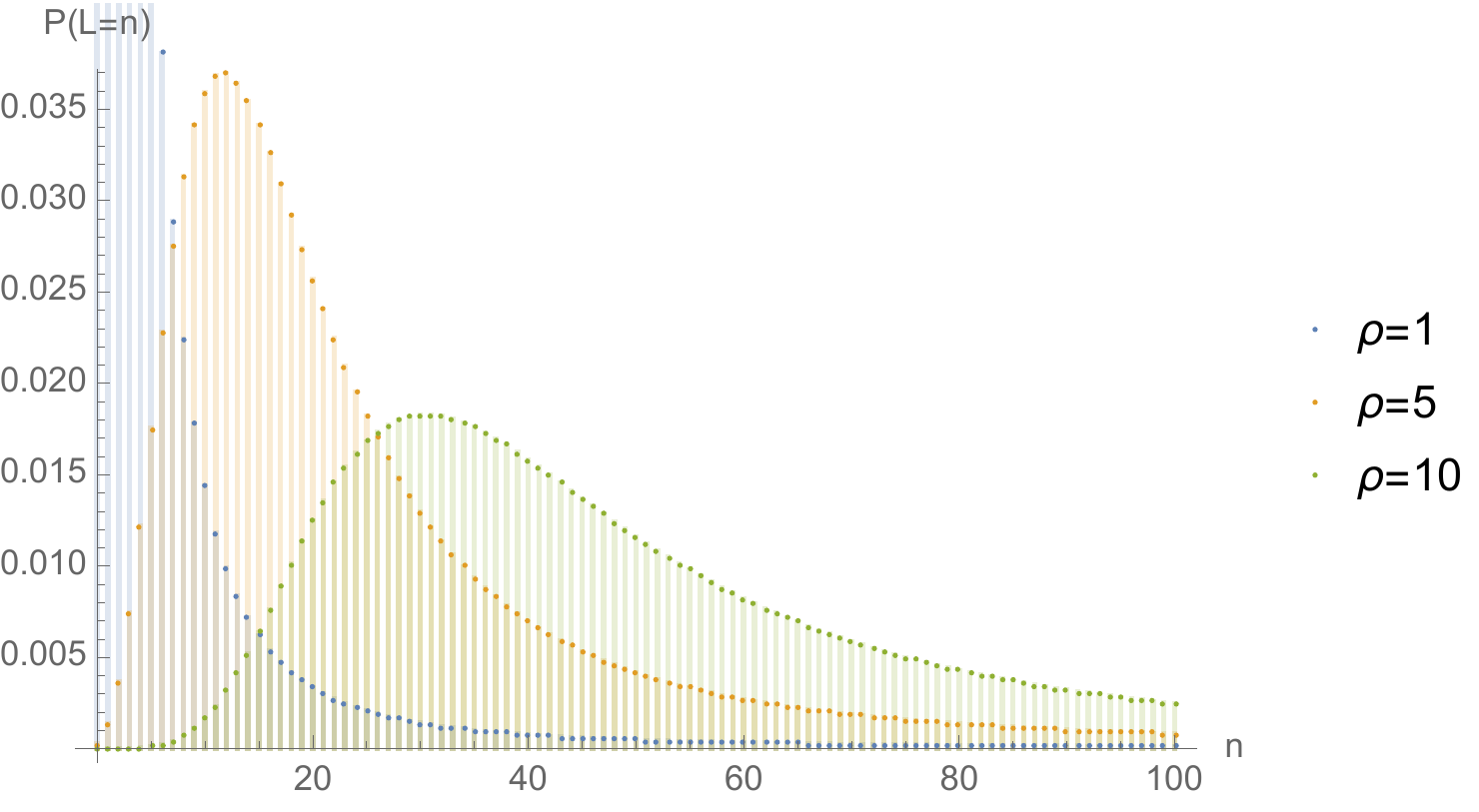}
\caption{Examples of the stationary distribution $P(L=n)$ in Battle \ref{example: Power-law batch size and exponential sojourn time; the infinite server wins}. The batch size $X$ is the discrete power law  distribution with $p=1$ and $E[X] = \infty$. The sojourn time is light-tailed, and the queue is stable.}
\label{fig:distribution.pdf}
\end{center}
\end{figure}
\begin{figure}[htbp]
\begin{center}
\includegraphics[width=0.4\textwidth]{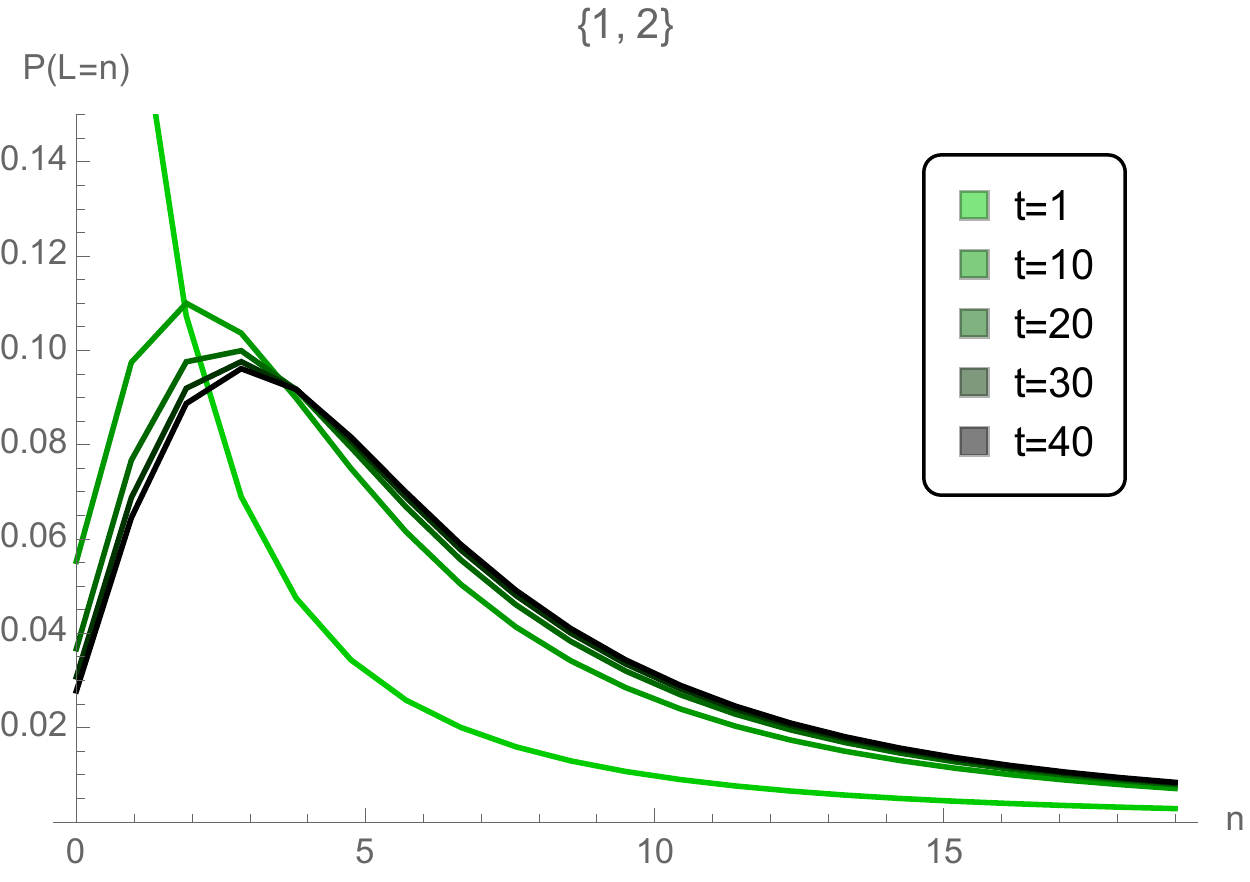}
\includegraphics[width=0.4\textwidth]{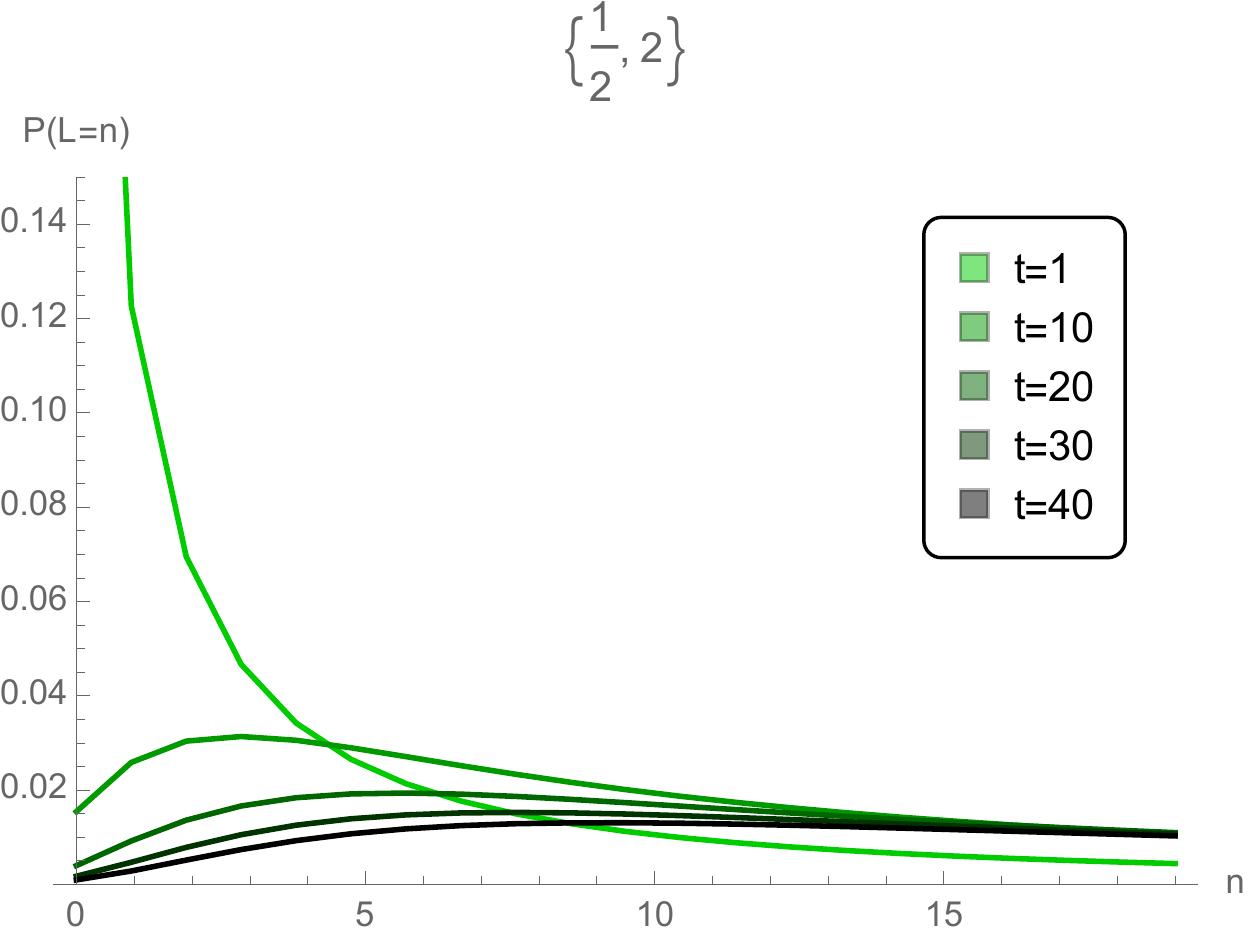}
\includegraphics[width=0.4\textwidth]{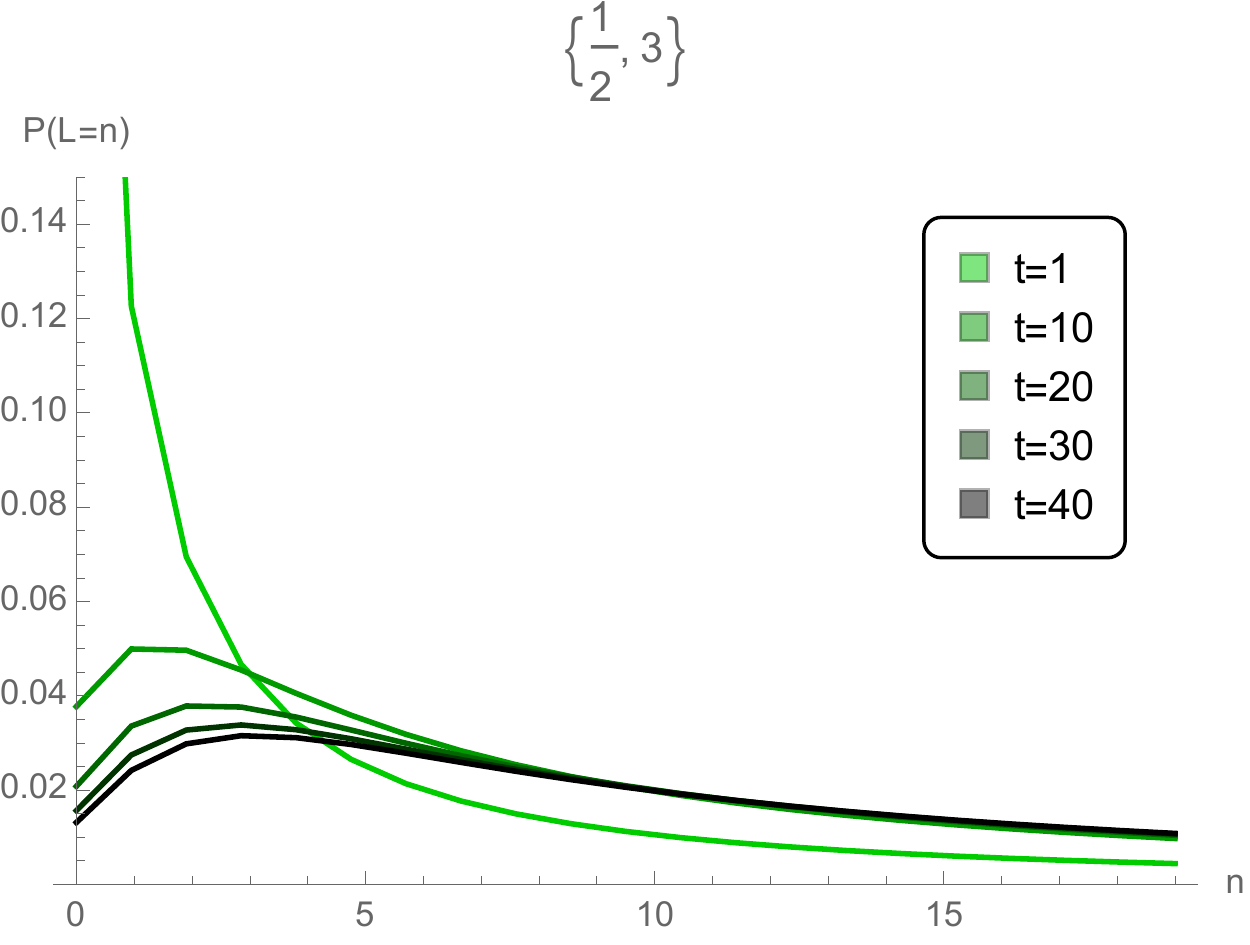}
\caption{The transient distribution $P(L(t)=n)$ of the infinite queues with the heavy-tailed arrivals $(p,q)=(1,2), (1/2,2)$ and $(1/2,3)$, and the batch arrival rate $\lambda =1$. }
\label{fig:transientPGF(1,2).pdf}
\end{center}
\end{figure}

Next, we analyze cases in which the batch size $X$ is a $p$-th order fractional power law discrete distribution and the sojourn time $S$ is also an independent $q$-th order power law discrete distribution. Their pair is called the heavy-tailed arrival $(p,q)$. In healthcare systems, hospital stays may be fitted to heavy-tailed distributions due to few severely ill patients \cite{S.:2018zc}.   By the balance criteria \eqref{ineq: Sx}, the infinite server queue with heavy-tailed arrival $(p,q)$ is stable when $pq>1$ and $q>1$. Thus, healthcare systems have to take into account statistical features such as the tail behavior of the cluster size and the hospital stay length distributions to ensure the stability of their systems, as can be seen in the following battle examples.

\begin{battle}[the heavy-tailed arrival $(1,2)$; the infinite server wins]
Consider the heavy-tailed arrival $(p,q)=(1,2)$. The sojourn time is heavy-tailed, but relatively mild ($E[S^{2}] = \infty$ but $E[S] < \infty$). Thus, the queue is stable, and
\begin{align}
E\left[z^{L}\right] = \exp\left[\lambda \sum_{k=1}^{\infty}  \frac{\alpha_{k}(z)-1}{\alpha_{k}(z)}\log (1-\alpha_{k}(z)) \right]. 
\end{align}
 where $\alpha_{k}(z) = 1-(1-z) P(S\geq k) = 1-\frac{2(1-z)}{k(k+1)}$.
The top figure of Figure \ref{fig:transientPGF(1,2).pdf} shows the transient behavior of the probability distribution converging to this steady-state distribution.
\end{battle}

\begin{battle}[the heavy-tailed arrival $(1/2,2)$; the explosive demand wins]\label{battle:the explosive demand wins}
Consider the heavy-tailed arrival $(p,q)=(1/2,2)$ that satisfies $E[\log X] < \infty, E[\sqrt X] = \infty$ and $E[S] < \infty, E[S^{2}] = \infty$. In this case, the demand surge is too explosive, so the system is overwhelmed by the demand, and it is unstable in the sense that $P(L=\infty)=1$. Since $\sin ^{-1}(x) \to \pi/2$ as $x \to 1$, we have
\begin{align}
E\left[\max_{1\leq i \leq X}S_{i}\right] = \sum_{k=1}^{\infty}\frac{\sin ^{-1}\left([1-1/\{(k(k+1)\}]^{1/2}\right)}{ \{k(k+1)-1\}^{1/2}} = \infty. 
\end{align}
Thus, by Theorem \ref{thm: ESX finite}, $P(L=\infty)=1$. The middle figure of Figure \ref{fig:transientPGF(1,2).pdf} shows the transient behavior of the probability distribution $P(L(t)=n)$, and shows that $L$ is escaping to infinity. Thus, no matter how fast we reinforce the capacity in order to realize the idealized healthcare system, collapse is inevitable.
\end{battle}

\begin{battle}[the heavy-tailed arrival $(1/2,3)$; the infinite server wins]
Consider the heavy-tailed arrival $(p,q)=(1/2,3)$. The batch size fluctuates severely as in the previous example, but the sojourn time is more gentle. The queue is stable, because $pq>1$, and we have
\begin{align}
E\left[z^{L}\right] = \exp\left[\lambda \sum_{k=1}^{\infty} \left(\frac{1-\alpha_{k}(z)}{\alpha_{k}(z)}\right)^{1/2}\sin ^{-1}\left(\sqrt{\alpha_{k}(z)}\right)\right].
\end{align}
where $\alpha_{k}(z) = 1-(1-z) P(S\geq k) = 1-6(1-z)/\{k(k+1)(k+2)\}$. The bottom figure of Figure \ref{fig:transientPGF(1,2).pdf} shows that $L$ does not escape to infinity, even though the arrival is heavy-tailed in both the batch size $X$ and the sojourn time $S$.
\end{battle}

\appendix
\section{Stability Criteria}\label{Appendix: Stability Criteria}

It is known that a $G/G/\infty$ queue \footnote{the symbol $A/B/c$ is called the Kendall symbol of the queuing system, with the arrival process $A$, the service demand $B$, and the number of servers $c$. $G/G/\infty$ means that both the arrival process and the service demand are general and dependent, and have infinitely many servers. On the other hand, $M/M/\infty$ means that the arrival process is Poisson process and the service demand is independent exponential distribution.} with infinitely many servers to process a stream of customers (no batch arrival) with generally-distributed sojourn time is stable when the arrival rate and the expectation of sojourn times of the customers are both finite \cite[p133]{baccelli2013elements}. Few studies \cite{Cong:1994fk,Yajima:2017fv} have extended this fundamental result to the case when the batches of customers arrive as a Poisson process. Even when the arrival is so heavy-tailed that the expected batch size is infinite, the infinite server can process the demand, and the queue is stable, given the batch size satisfies $E[\log X] < \infty$. Here, we obtain a new criterion to check the stability of the batch arrival $G^{X}/G/\infty$ queue ($X$ represents the batch arrival), when the sojourn times of customers may be correlated.  
%

\begin{figure}[tbp]
\begin{center}
\includegraphics[width=0.4\textwidth]{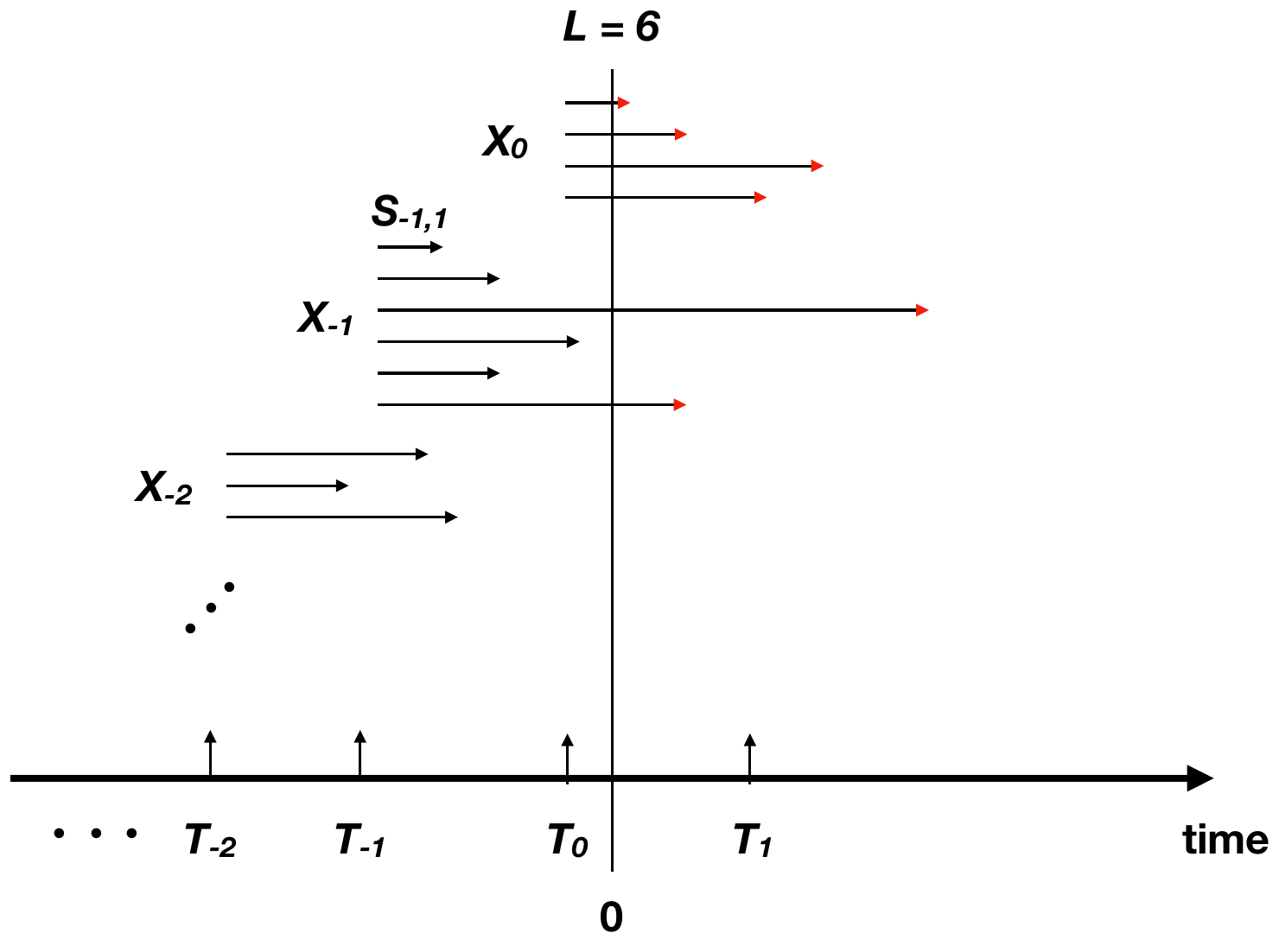}
\caption{Examples of a sample path of the infinite server queue and $L$, the number of customer in the system at time $0$.  The arrows indicate the sojourn time of each customer who arrived in the batch.  The red-end arrows indicate that the customers is in the system at time $0$, i.e. they are in $A_{-k,i}$.  Thus, $L = 6$ in this example.}
\label{fig:L.pdf}
\end{center}
\end{figure}
\begin{figure}[tbp]
\begin{center}
\includegraphics[width=0.4\textwidth]{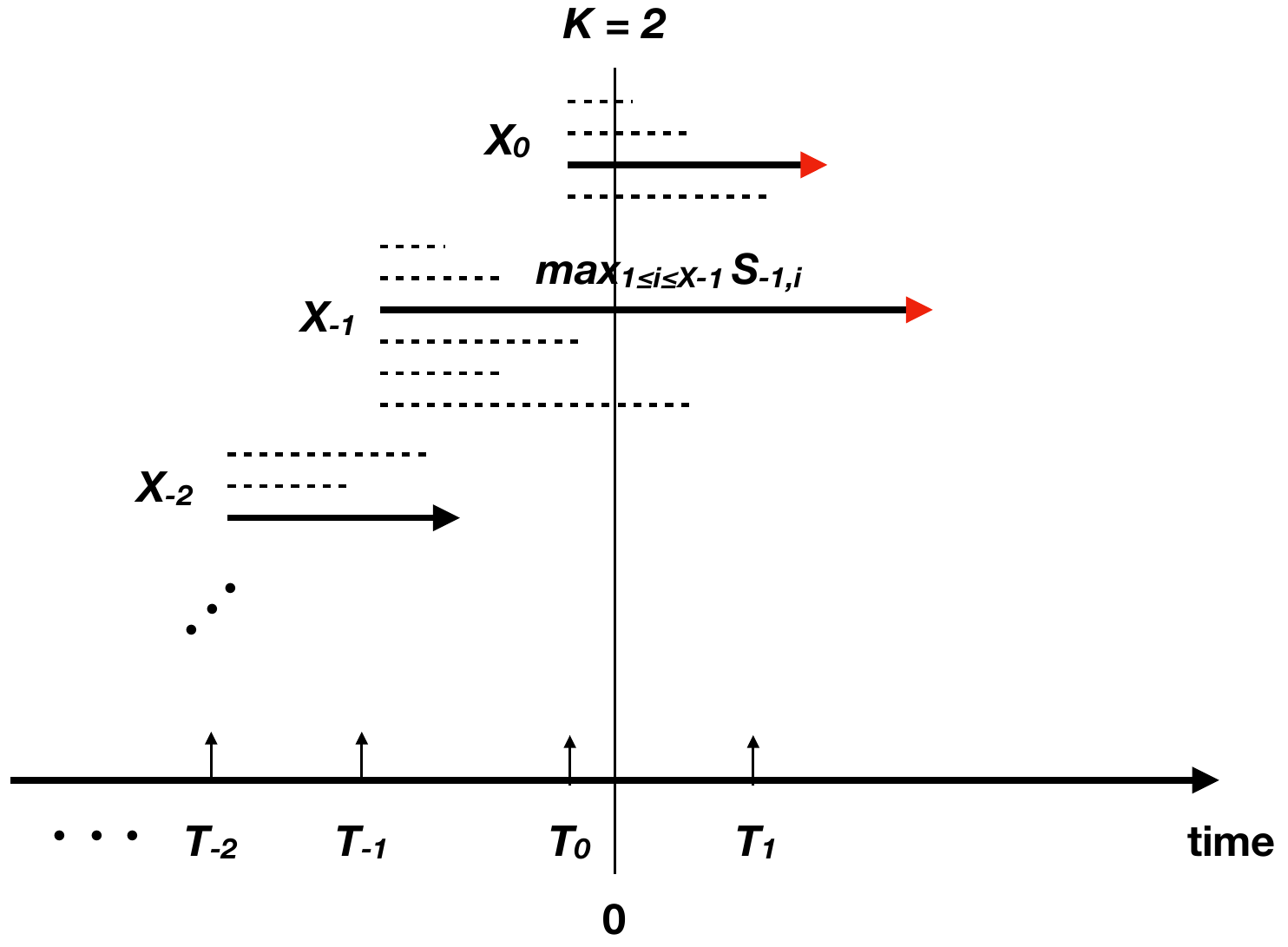}
\caption{Examples of a sample path of super customers.  The thick arrows indicate the sojourn time of super customers.  The red-end arrows indicate that the super customers who are in the system at time $0$, i.e. they are in $B_{-k}$.  Thus, $K = 2$ in this example.}
\label{fig:K.pdf}
\end{center}
\end{figure}

As discussed in Section \ref{section:Infinite server queue and its stability} (see also Figure \ref{fig:L.pdf}), $L$, the number of customers in the system at time $0$ can be expressed as
\begin{align}\label{eq: L }
L = \sum_{k=0}^{\infty}\sum_{i=1}^{X_{-k}}1_{A_{-k,i}}.
\end{align}
Let $B_{-k}= \cup_{i=1}^{X_{^{k}}}A_{-k,i} = \{\max_{1\leq i\leq X_{-k}}S_{-k,i}>-T_{-k}\}$ be the event when least one customer in the batch arriving at $T_{-k}$ is still in the system at time $0$. Then, 
\begin{align}\label{eq:def of K}
K = \sum_{k=0}^{\infty}1_{B_{-k}}, 
\end{align}
is the number of super-customers whose sojourn times are assigned as $\max_{1\leq i \leq X_{k}}S_{i}$ (see Figure \ref{fig:K.pdf}). Using the monotone convergence theorem, we have 
\begin{align}\label{eq:EK}
E[K] = E\left[\sum_{k=0}^{\infty}1_{B_{-k}}\right] = \sum_{k=0}^{\infty} P(B_{k}), 
\end{align}
If $E[K] < \infty$, we have $P(K=\infty) = P(B_{-k} \text{ i.o.}) = 0$ by the Borel-Cantelli lemma. Indeed, if $P(K=\infty) > 0$, then $E[K]$ should be infinite.

Consider a $G/G/\infty$ queue with the arrival of super-customers at $T_{n}$, and their sojourn time with $ \max_{1\leq i \leq X_{k}}S_{i}$. If $ E[\max_{1\leq i \leq X_{k}}S_{i}] < \infty$, then by Little's formula (see \cite{baccelli2013elements,wolff1989stochastic,kleinrock1975queueing}, for the $G/G/\infty$ queue,
\begin{align}
E[K] = \lambda E\left[\max_{1\leq i \leq X}S_{i}\right] < \infty, 
\end{align}
Thus, $P(K<\infty) = 1$ (this is a known result of the stability for the $G/G/\infty$ queue, \cite[p133]{baccelli2013elements}). By \eqref{eq: L }, 
\begin{align}
L \leq \sum_{k=0}^{\infty}1_{B_{-k}}X_{-k}, 
\end{align}
Since $K$ is finite, the sum on the right-hand side is indeed a finite sum of finite quantities. Thus, $L<\infty$. Consequently, we have the following result: 
\begin{theorem}\label{theorem: stability}
If the batch arrival rate $\lambda$ and $E[\max_{1\leq i \leq X_{k}}S_{i}]$ are finite, then $G^{X}/G/\infty$ queue is stable in the sense of $P(L<\infty) = 1$.
\end{theorem}

\section{The balance criteria of stability}\label{Appendix:The Balance Criteria of Stability}
Using the Holder inequality, we extend the idea in \cite{Devroye:1979fk} to obtain an estimation of the expectation of $S_{(X)}=max_{1\leq i \leq X}S_{i}$, as in \eqref{ineq: Sx}. We assume $S_{(X)}=0$ when $X=0$. First, we prove a general result about the bound of the expectation of the maximum values of correlated random variables.


\begin{lemma}\label{lemma: max}
Let $Y_{1}, Y_{2}, \dots Y_{n}$ be possibly correlated non-negative random variables, all of which have the same marginal distribution denoted by a random variable $Y$.
If $E[Y^{p}]< \infty$ for some $p \in [1, \infty)$, then $Y_{(n)} = \max_{i=1,2, \dots, n}Y_{i}$ has the following estimation: 
\begin{align}
E[Y_{(n)}] \leq \{n E[Y^{p}]\}^{1/p}, 
\end{align}
\end{lemma}

\begin{proof}
Assuming that $Y_{1}, Y_{2}, \dots Y_{n}$ are defined in the probability space $(\Omega, P)$, we define the set $A_{i} = \{Y_{i}>\max(Y_{1}, \dots, Y_{i-1}), Y_{i} = Y_{(n)}\}$. This means that $A_{i}$, $Y_{i}$ is the first among possibly multiple random variables attaining the value $Y_{(n)}$. Since $A_{1},\dots, A_{n}$ are disjoint, and $\cup_{i=1}^{n} A_{i} = \Omega$,
\begin{align}\label{eq: E Yn}
E[Y_{(n)}] = \sum_{i=1}^{n}E[Y_{i}1_{A_{i}}]= \sum_{i=1}^{n}E[Y_{i}E[1_{A_{i}}|Y_{i}].
\end{align}
Let $q=p/(p-1) \geq 1$, so that $1/p+1/q =1$. Since $E[1_{A_{i}}|Y_{i}]\leq 1$, we have $\{E[1_{A_{i}}|Y_{i}]\}^{q}\leq E[1_{A_{i}}|Y_{i}]$. Thus,
\begin{align}
E\left[E[1_{A_{i}}|Y_{i}]^{q} \right] \leq P(A_{i}), 
\end{align}
Applying the Holder inequality, we have 
\begin{align}
E[Y_{i}E[1_{A_{i}}|Y_{i}]] &\leq \left\{E[Y_{i}^{p}] \right\}^{1/p}\left\{ E[1_{A_{i}}|Y_{i}]^{q}\right\}^{1/q}. \\
&\leq \left\{E[Y^{p}] \right\}^{1/p} \left\{ P(A_{i})\right\}^{1/q}, 
\end{align}
Using this in \eqref{eq: E Yn} and applying the Jensen's inequality, we have
\begin{align}
E[Y_{(n)}] &\leq \left\{E[Y^{p}] \right\}^{1/p}\sum_{i=1}^{n} \left\{ P(A_{i})\right\}^{1/q}. \\
&\leq \left\{E[Y^{p}] \right\}^{1/p} n \left\{ \sum_{i=1}^{n} \frac{1}{n}P(A_{i})\right\}^{1/q}. \\
&= \left\{E[Y^{p}] \right\}^{1/p}n^{1/p}, 
\end{align}
\end{proof}

\begin{theorem}\label{thm: stability pq >1}
Consider a $G^{X}/G/\infty$ queue with possibly correlated sojourn times $S_{1}, S_{2}, \dots S_{(X)}$ in the batch. If $S$ and $X$ are independent, $E[S^{p}]<\infty$, and $E[X^{1/p}]< \infty$ for some $p\in [1,\infty)$, the queue is stable; that is, $P(L< \infty)=1$. 
\end{theorem}

\begin{proof}
By Theorem \ref{theorem: stability}, it is enough to show $E[S_{(X)}] < \infty$.   Since $E[S^{p}]< \infty$, using Lemma \ref{lemma: max} and the independence of $X$ and $S$, we have
\begin{align}
E[S_{(X)}|X=n] \leq \{n E[S^{p}]\}^{1/p}, 
\end{align}
Unconditioning on $X$ gives
\begin{align}
E[S_{(X)}] \leq \{ E[S^{p}]\}^{1/p} E[X^{1/p}]. 
\end{align}
which proves our claim.
\end{proof}

\section{Infinite server queues with Poisson batch arrival}\label{Appendix:Infinite Server Queues with Poisson Batch Arrival}
In this section, we assume that batches of customers arrive as a Poisson process with the rate $\lambda$.  All customers in the same batch should depart by the time $S_{(X)} = \max_{1\leq i \leq X}S_{i}$.  Let $L(t)$ be the number of customers in the system at time $t$ starting from $L(0)=0$, and let $L= \lim_{t\to\infty}L(t)$ be its weak limit.

\begin{lemma}\label{thm: transient pgf}
The probability generating function of $L(t)$ is given by
\begin{align}\label{eq:transient pgf}
E[z^{L(t)}]=\exp \left[ - \lambda \left\{ \int_{0}^{t}\left(1- E\left[z^{M(s)}\right]\right)ds \right\}\right].
\end{align}
for $|z| \leq 1$ where $M(t)=\sum_{i=1}^{X}1_{\{S_{i}>t\}}$.  
\end{lemma}
\begin{proof}
Let $M_{j}(t)$ be the number of customers remaining in the system at time $t$ from the $j$-th batch that arrives at time $T_{j}$. The batch arrival time $T_{j}=s$, $M_{j}(t)$ has the same distribution as $M(t-s)$. Since $T_{j}$ is uniformly distributed on the interval $(0,t)$ when conditioning on $N(t)$ (the number of batches arriving up to time $t$), we have
\begin{align}
E[z^{M_{j}(t)}|N(t)]&= \frac{1}{t}\int_{0}^{t}E\left[z^{M_{j}(t)}\Big|N(t),T_{j}=s\right]ds, \\
&=\frac{1}{t}\int_{0}^{t}E\left[z^{M(t-s)}\right]ds, \\
&=\frac{1}{t}\int_{0}^{t}E\left[z^{M(s)}\right]ds,
\end{align}
for $1\leq j \leq N(t)$.  Since $N(t)$ is a Poisson process, we have
\begin{align}
E[z^{L(t)}]&= E\left[ z^{\sum_{j=1}^{N(t)}M_{j}(t)}\right], \\
&= \exp \left[ \lambda t \left\{E\left[ z^{M_{j}(t)}\right] -1\right\}\right].
\\&=\exp \left[ - \lambda \left\{ \int_{0}^{t}\left(1- E\left[z^{M(s)}\right]\right)ds \right\}\right].
\end{align}
\end{proof}

\begin{theorem}\label{thm: ESX finite}
The $M^{X}/G/\infty$ queue with correlated sojourn times in the batch is stable in the sense that $P(L<\infty) = 1$, if and only if $E[S_{(X)} ] <\infty$.

In this case, the probability generating function of $L$ is given by
\begin{align}\label{eq: stationary pgf}
E\left[z^{L}\right]=\exp \left[ - \lambda \left\{ \int_{0}^{\infty}\left(1- E\left[z^{M(s)}\right]\right)ds \right\}\right]. 
\end{align}
for $|z| \leq 1,$ where $M(t)=\sum_{i=1}^{X}1_{\{S_{i}>t\}}$. Furthermore, when the sojourn times are independent and identically distributed, 
\begin{align}\label{eq: stationary pgf for independent sojourn time}
E\left[z^{M(s)}\right]= \phi(z P(S>s) + P(S\leq s)]. 
\end{align}
where $\phi(z) = E[z^{X}]$ is the generating function of the batch size $X$.
\end{theorem}

Here, we also provide examples of the generating function of the $p$-th order fractional power law distributions \eqref{eq: Power Law Type}:
\begin{align}\label{eq: examples of pgf}
E[z^{X}] =
{\small \begin{cases}, 
1-2(1-z)\{1 - (1-z) \log(1 - z)\}/z^2  &\text{ $(p =2)$},\\
1 - (1 - 1/z) \log(1 - z)  &\text{ $(p =1)$},\\
1- \left(\frac{1-z}{z}\right)^{1/2}\sin ^{-1}\left(\sqrt{z}\right) &\text{ $(p =1/2)$},
\end{cases}}
\end{align}
which are used in Section \ref{section:Battles against explosive demand surges}.

\begin{proof}
Taking $t\to\infty$ formally in \eqref{eq:transient pgf}, we have the generating function of the weak limit $L$ as
\begin{align}\label{eq:P(L)}
E\left[z^{L}\right]=\exp \left[ - \lambda \left\{ \int_{0}^{\infty}\left(1- E\left[z^{M(s)}\right]\right)ds \right\}\right],
\end{align}
which is an extension of the equation (6) in \cite{Cong:1994fk} to the case of correlated sojourn times in the batch.  

First, let us assume that $E[S_{(X)} ] <\infty$. We then check the weak limit of $L$ as a proper random variable satisfying $P(L<\infty) = 1$. It is easy to see that $1- E\left[z^{M(s)}\right] \leq P(M(s)>0) = P(S_{(X)}>s)$ for all $s \in [0,\infty)$. Furthermore, $P(S_{(X)}>s)$ is integrable on $[0, \infty)$, because $\int_{0}^{\infty}P(S_{(X)}>s)ds = E[S_{(X)}] <\infty$. Using the dominant convergence theorem, we can change the order of the limit and the integral to have
\begin{align}
P(L<\infty) &= \sum_{n=0}^{\infty}P(L=n) = \lim_{z\to 1}E\left[z^{L}\right], \\
&=\exp \left[ - \lambda \left\{ \int_{0}^{\infty}\left(1- \lim_{z\to 1}E\left[z^{M(s)}\right]\right)ds \right\}\right]=1.
\end{align}
 Conversely, we assume $P(L<\infty) =1$. Then, the integral in \eqref{eq:P(L)} should be equal to zero when $z=1$. Since
\begin{align}
E\left[z^{M(s)}\right] &\leq \sum_{n=1}^{\infty}zP(M(s)=n)+P(M(s)=0), \\
&= 1 - (1-z) P(M(s)>0)
\end{align}
for $0 \leq z \leq 1$, we have 
\begin{align}
0&=\lim_{z\to 1}\int_{0}^{\infty}\left(1- E\left[z^{M(s)}\right]\right)ds, \\
&\geq \lim_{z\to 1}(1-z)\int_{0}^{\infty}P(M(s)>0)ds, \\
&= \lim_{z\to 1}(1-z)E[S_{(X)} ]\geq 0. 
\end{align}
which leads to the conclusion that $E[S_{(X)} ] <\infty$.
\end{proof}

\begin{remark}
As can be seen from Theorem \ref{thm: stability pq >1}, $E[S_{(X)} ] <\infty$ is sufficient for the stability of general arrival processes, but the converse is not always true. However, if the arrival of batches is restricted to Poisson processes, $P(L < \infty) =1$ leads to $E[S_{(X)} ] <\infty$, which can be alternatively proved by the reverse Borel-Cantelli lemma, because $B_{-k}$ in Appendix \ref{Appendix: Stability Criteria} are independent. However, the generating function representation given by \eqref{eq:transient pgf} and \eqref{eq: stationary pgf} provide the detailed dynamics of the system.
\end{remark}

It is known that when the sojourn times are independent exponential random variables \cite{Cong:1994fk} or light-tailed random variables \cite{Yajima:2017fv}, $E[\log X]<\infty$ is a necessary and sufficient condition for stability. Here, we provide an alternative proof for this case below. However, the condition $E[\log X]<\infty$ no longer guarantees stability when the demand is too explosive (see Battle \ref{battle:the explosive demand wins}).  

\begin{corollary}\label{corollary:log bound}
Consider an $M^{X}/M/\infty$ queue with independent exponential sojourn times $S_{1}, S_{2}, \dots S_{X}$ with mean $1/\mu$ in the batch. The queue is stable, that is, $P(L< \infty)=1$, if and only if $E[\log X] < \infty$ \cite{Cong:1994fk,Yajima:2017fv}.  

 In this case, the probability generating function of the stationary distribution $L$ is given by
\begin{align}\label{eq: M/M case}
E\left[z^{L}\right]=\exp \left[ - \frac{\lambda}{\mu} \sum_{n=1}^{\infty}\frac{1-z^{n}}{n}P(X\geq n)\right], 
\end{align}
\end{corollary}

\begin{proof}
Since the sojourn times are independent exponential times, we have $E[S_{(X)}] = E\left[\sum_{i=1}^{X}1/i\right]$, and $\sum_{i=1}^{X}1/i$ is bounded as
\begin{align}
\log X \leq \sum_{i=1}^{X}\frac{1}{i} \leq 1 + \log X,
\end{align}
which proves the first part using Theorem \ref{thm: ESX finite}. By conditioning on $X=k$ and using \eqref{eq: stationary pgf for independent sojourn time}, we have
\begin{align}
 \int_{0}^{\infty}&\left(1- E\left[z^{M(s)}\right]\right)ds \\
 &= \int_{0}^{\infty}\sum_{k=1}^{\infty}\left[1- \{ze^{-\mu s}+ (1-e^{-\mu s})\}^{k} \right]P(X=k)ds,  \\
 &= \frac{1}{\mu}\sum_{k=1}^{\infty}P(X=k)\int_{z}^{1}\frac{(1-u^{k})}{1-u}du, \\
 &=\frac{1}{\mu}\sum_{k=1}^{\infty}P(X=k)\sum_{n=1}^{k}\frac{1-z^{n}}{n}, \\
 &=\frac{1}{\mu}\sum_{n=1}^{\infty}\frac{1-z^{n}}{n}P(X\geq n), 
\end{align}
which proves the second part by combining with \eqref{eq: stationary pgf}.
\end{proof}

We now consider the case where the batch size $X$ is the $p$-th order fractional power law distribution. By \eqref{eq: M/M case}, the probability generating function of the stationary distribution $L$ is given by
\begin{align}
E\left[z^{L}\right]=\exp \left[ - \rho \sum_{n=1}^{\infty}\frac{1-z^{n}}{n} \prod_{i=1}^{n-1}\frac{i}{p+i}\right]. 
\end{align}
for $p \in (0,\infty)$, where $\rho = \lambda/\mu$. Especially, when $p=1$, even though $E[X]$ is infinite, the queue is stable, and
\begin{align}
E\left[z^{L}\right]=\exp \left[ - \rho \sum_{n=1}^{\infty}\frac{1-z^{n}}{n^{2}} \right].
\end{align}


\bibliographystyle{apsrev4-2}
\bibliography{/Users/toyo/Box/Public_Share/References/2018,/Users/toyo/Box/Public_Share/References/2019,/Users/toyo/Box/Public_Share/References/2020}

\begin{thebibliography}{17}%
\makeatletter
\providecommand \@ifxundefined [1]{%
 \@ifx{#1\undefined}
}%
\providecommand \@ifnum [1]{%
 \ifnum #1\expandafter \@firstoftwo
 \else \expandafter \@secondoftwo
 \fi
}%
\providecommand \@ifx [1]{%
 \ifx #1\expandafter \@firstoftwo
 \else \expandafter \@secondoftwo
 \fi
}%
\providecommand \natexlab [1]{#1}%
\providecommand \enquote  [1]{``#1''}%
\providecommand \bibnamefont  [1]{#1}%
\providecommand \bibfnamefont [1]{#1}%
\providecommand \citenamefont [1]{#1}%
\providecommand \href@noop [0]{\@secondoftwo}%
\providecommand \href [0]{\begingroup \@sanitize@url \@href}%
\providecommand \@href[1]{\@@startlink{#1}\@@href}%
\providecommand \@@href[1]{\endgroup#1\@@endlink}%
\providecommand \@sanitize@url [0]{\catcode `\\12\catcode `\$12\catcode
  `\&12\catcode `\#12\catcode `\^12\catcode `\_12\catcode `\%12\relax}%
\providecommand \@@startlink[1]{}%
\providecommand \@@endlink[0]{}%
\providecommand \url  [0]{\begingroup\@sanitize@url \@url }%
\providecommand \@url [1]{\endgroup\@href {#1}{\urlprefix }}%
\providecommand \urlprefix  [0]{URL }%
\providecommand \Eprint [0]{\href }%
\providecommand \doibase [0]{https://doi.org/}%
\providecommand \selectlanguage [0]{\@gobble}%
\providecommand \bibinfo  [0]{\@secondoftwo}%
\providecommand \bibfield  [0]{\@secondoftwo}%
\providecommand \translation [1]{[#1]}%
\providecommand \BibitemOpen [0]{}%
\providecommand \bibitemStop [0]{}%
\providecommand \bibitemNoStop [0]{.\EOS\space}%
\providecommand \EOS [0]{\spacefactor3000\relax}%
\providecommand \BibitemShut  [1]{\csname bibitem#1\endcsname}%
\let\auto@bib@innerbib\@empty
\bibitem [{\citenamefont {On~Kwok}\ \emph {et~al.}(2020)\citenamefont
  {On~Kwok}, \citenamefont {Hin~Chan}, \citenamefont {Huang}, \citenamefont
  {Cheong~Hui}, \citenamefont {Anantharajah~Tambyah}, \citenamefont {In~Wei},
  \citenamefont {Kwan~Chau}, \citenamefont {Shan~Wong},\ and\ \citenamefont
  {Tze~Tang}}]{On-Kwok:2020bs}%
  \BibitemOpen
  \bibfield  {author} {\bibinfo {author} {\bibfnamefont {K.}~\bibnamefont
  {On~Kwok}}, \bibinfo {author} {\bibfnamefont {H.~H.}\ \bibnamefont
  {Hin~Chan}}, \bibinfo {author} {\bibfnamefont {Y.}~\bibnamefont {Huang}},
  \bibinfo {author} {\bibfnamefont {D.~S.}\ \bibnamefont {Cheong~Hui}},
  \bibinfo {author} {\bibfnamefont {P.}~\bibnamefont {Anantharajah~Tambyah}},
  \bibinfo {author} {\bibfnamefont {W.}~\bibnamefont {In~Wei}}, \bibinfo
  {author} {\bibfnamefont {P.~Y.}\ \bibnamefont {Kwan~Chau}}, \bibinfo {author}
  {\bibfnamefont {S.~Y.}\ \bibnamefont {Shan~Wong}},\ and\ \bibinfo {author}
  {\bibfnamefont {J.~W.}\ \bibnamefont {Tze~Tang}},\ }\href
  {https://doi.org/10.1016/j.jhin.2020.05.027} {\bibfield  {journal} {\bibinfo
  {journal} {The Journal of hospital infection}\ }\textbf {\bibinfo {volume}
  {105}},\ \bibinfo {pages} {682} (\bibinfo {year} {2020})}\BibitemShut
  {NoStop}%
\bibitem [{\citenamefont {Iritani}\ \emph {et~al.}(2020)\citenamefont
  {Iritani}, \citenamefont {Okuno}, \citenamefont {Hama}, \citenamefont {Kane},
  \citenamefont {Kodera}, \citenamefont {Morigaki}, \citenamefont {Terai},
  \citenamefont {Maeno},\ and\ \citenamefont {Morimoto}}]{Iritani:2020xr}%
  \BibitemOpen
  \bibfield  {author} {\bibinfo {author} {\bibfnamefont {O.}~\bibnamefont
  {Iritani}}, \bibinfo {author} {\bibfnamefont {T.}~\bibnamefont {Okuno}},
  \bibinfo {author} {\bibfnamefont {D.}~\bibnamefont {Hama}}, \bibinfo {author}
  {\bibfnamefont {A.}~\bibnamefont {Kane}}, \bibinfo {author} {\bibfnamefont
  {K.}~\bibnamefont {Kodera}}, \bibinfo {author} {\bibfnamefont
  {K.}~\bibnamefont {Morigaki}}, \bibinfo {author} {\bibfnamefont
  {T.}~\bibnamefont {Terai}}, \bibinfo {author} {\bibfnamefont
  {N.}~\bibnamefont {Maeno}},\ and\ \bibinfo {author} {\bibfnamefont
  {S.}~\bibnamefont {Morimoto}},\ }\href
  {https://doi.org/doi:10.1111/ggi.13973} {\bibfield  {journal} {\bibinfo
  {journal} {Geriatrics \& Gerontology International}\ }\textbf {\bibinfo
  {volume} {20}},\ \bibinfo {pages} {715} (\bibinfo {year} {2020})}\BibitemShut
  {NoStop}%
\bibitem [{\citenamefont {Easley}\ and\ \citenamefont
  {Kleinberg}(2010)}]{easley2010networks}%
  \BibitemOpen
  \bibfield  {author} {\bibinfo {author} {\bibfnamefont {D.}~\bibnamefont
  {Easley}}\ and\ \bibinfo {author} {\bibfnamefont {J.}~\bibnamefont
  {Kleinberg}},\ }\href@noop {} {\emph {\bibinfo {title} {Networks, crowds, and
  markets}}},\ Vol.~\bibinfo {volume} {8}\ (\bibinfo  {publisher} {Cambridge
  Univ Press},\ \bibinfo {year} {2010})\BibitemShut {NoStop}%
\bibitem [{\citenamefont {Levis}(2009)}]{levis2009winners}%
  \BibitemOpen
  \bibfield  {author} {\bibinfo {author} {\bibfnamefont {K.}~\bibnamefont
  {Levis}},\ }\href@noop {} {\emph {\bibinfo {title} {Winners and Losers:
  Creators and Casualties of the Age of the Internet}}}\ (\bibinfo  {publisher}
  {Atlantic Books Ltd},\ \bibinfo {year} {2009})\BibitemShut {NoStop}%
\bibitem [{\citenamefont {Porter}\ and\ \citenamefont {ilustraciones
  Gibbs}(2001)}]{porter2001strategy}%
  \BibitemOpen
  \bibfield  {author} {\bibinfo {author} {\bibfnamefont {M.~E.}\ \bibnamefont
  {Porter}}\ and\ \bibinfo {author} {\bibfnamefont {M.}~\bibnamefont
  {ilustraciones Gibbs}},\ }\href@noop {} {\  (\bibinfo {year}
  {2001})}\BibitemShut {NoStop}%
\bibitem [{\citenamefont {Prakash}\ \emph {et~al.}(2012)\citenamefont
  {Prakash}, \citenamefont {Beutel}, \citenamefont {Rosenfeld},\ and\
  \citenamefont {Faloutsos}}]{prakash2012winner}%
  \BibitemOpen
  \bibfield  {author} {\bibinfo {author} {\bibfnamefont {B.~A.}\ \bibnamefont
  {Prakash}}, \bibinfo {author} {\bibfnamefont {A.}~\bibnamefont {Beutel}},
  \bibinfo {author} {\bibfnamefont {R.}~\bibnamefont {Rosenfeld}},\ and\
  \bibinfo {author} {\bibfnamefont {C.}~\bibnamefont {Faloutsos}},\ }in\
  \href@noop {} {\emph {\bibinfo {booktitle} {Proceedings of the 21st
  international conference on World Wide Web}}}\ (\bibinfo {organization}
  {ACM},\ \bibinfo {year} {2012})\ pp.\ \bibinfo {pages}
  {1037--1046}\BibitemShut {NoStop}%
\bibitem [{\citenamefont {Baccelli}\ and\ \citenamefont
  {Br{\'e}maud}(2013)}]{baccelli2013elements}%
  \BibitemOpen
  \bibfield  {author} {\bibinfo {author} {\bibfnamefont {F.}~\bibnamefont
  {Baccelli}}\ and\ \bibinfo {author} {\bibfnamefont {P.}~\bibnamefont
  {Br{\'e}maud}},\ }\href@noop {} {\emph {\bibinfo {title} {Elements of
  queueing theory: Palm Martingale calculus and stochastic recurrences}}},\
  Vol.~\bibinfo {volume} {26}\ (\bibinfo  {publisher} {Springer Science \&
  Business Media},\ \bibinfo {year} {2013})\BibitemShut {NoStop}%
\bibitem [{\citenamefont {Wolff}(1989)}]{wolff1989stochastic}%
  \BibitemOpen
  \bibfield  {author} {\bibinfo {author} {\bibfnamefont {R.~W.}\ \bibnamefont
  {Wolff}},\ }\href@noop {} {\emph {\bibinfo {title} {Stochastic modeling and
  the theory of queues}}}\ (\bibinfo  {publisher} {Pearson College Division},\
  \bibinfo {year} {1989})\BibitemShut {NoStop}%
\bibitem [{\citenamefont {Kleinrock}(1975)}]{kleinrock1975queueing}%
  \BibitemOpen
  \bibfield  {author} {\bibinfo {author} {\bibfnamefont {L.}~\bibnamefont
  {Kleinrock}},\ }\href@noop {} {\emph {\bibinfo {title} {Queueing systems.
  Volume I: theory}}}\ (\bibinfo  {publisher} {Wiley New York},\ \bibinfo
  {year} {1975})\BibitemShut {NoStop}%
\bibitem [{\citenamefont {Mandelbaum}\ \emph {et~al.}(2007)\citenamefont
  {Mandelbaum}, \citenamefont {Hlynka},\ and\ \citenamefont
  {Brill}}]{Mandelbaum2007}%
  \BibitemOpen
  \bibfield  {author} {\bibinfo {author} {\bibfnamefont {M.}~\bibnamefont
  {Mandelbaum}}, \bibinfo {author} {\bibfnamefont {M.}~\bibnamefont {Hlynka}},\
  and\ \bibinfo {author} {\bibfnamefont {P.~H.}\ \bibnamefont {Brill}},\ }\href
  {https://doi.org/10.1007/s11750-007-0018-z} {\bibfield  {journal} {\bibinfo
  {journal} {TOP}\ }\textbf {\bibinfo {volume} {15}},\ \bibinfo {pages} {281}
  (\bibinfo {year} {2007})}\BibitemShut {NoStop}%
\bibitem [{\citenamefont {Dorogovtsev}\ \emph {et~al.}(2000)\citenamefont
  {Dorogovtsev}, \citenamefont {Mendes},\ and\ \citenamefont
  {Samukhin}}]{dorogovtsev2000structure}%
  \BibitemOpen
  \bibfield  {author} {\bibinfo {author} {\bibfnamefont {S.~N.}\ \bibnamefont
  {Dorogovtsev}}, \bibinfo {author} {\bibfnamefont {J.~F.~F.}\ \bibnamefont
  {Mendes}},\ and\ \bibinfo {author} {\bibfnamefont {A.~N.}\ \bibnamefont
  {Samukhin}},\ }\href@noop {} {\bibfield  {journal} {\bibinfo  {journal}
  {Physical review letters}\ }\textbf {\bibinfo {volume} {85}},\ \bibinfo
  {pages} {4633} (\bibinfo {year} {2000})}\BibitemShut {NoStop}%
\bibitem [{\citenamefont {Durrett}(2007)}]{durrett2007random}%
  \BibitemOpen
  \bibfield  {author} {\bibinfo {author} {\bibfnamefont {R.}~\bibnamefont
  {Durrett}},\ }\href@noop {} {\emph {\bibinfo {title} {Random graph
  dynamics}}},\ Vol.~\bibinfo {volume} {20}\ (\bibinfo  {publisher} {Cambridge
  Univ Pr},\ \bibinfo {year} {2007})\BibitemShut {NoStop}%
\bibitem [{\citenamefont {Devroye}(1979)}]{Devroye:1979fk}%
  \BibitemOpen
  \bibfield  {author} {\bibinfo {author} {\bibfnamefont {L.~P.}\ \bibnamefont
  {Devroye}},\ }\bibfield  {booktitle} {\emph {\bibinfo {booktitle}
  {Mathematics of Operations Research}},\ }\href
  {https://doi.org/10.1287/moor.4.4.441} {\bibfield  {journal} {\bibinfo
  {journal} {Mathematics of Operations Research}\ }\textbf {\bibinfo {volume}
  {4}},\ \bibinfo {pages} {441} (\bibinfo {year} {1979})}\BibitemShut {NoStop}%
\bibitem [{\citenamefont {Cong}(1994)}]{Cong:1994fk}%
  \BibitemOpen
  \bibfield  {author} {\bibinfo {author} {\bibfnamefont {T.~D.}\ \bibnamefont
  {Cong}},\ }\href {https://doi.org/DOI: 10.2307/3215256} {\bibfield  {journal}
  {\bibinfo  {journal} {Journal of Applied Probability}\ }\textbf {\bibinfo
  {volume} {31}},\ \bibinfo {pages} {280} (\bibinfo {year} {1994})}\BibitemShut
  {NoStop}%
\bibitem [{\citenamefont {Yajima}\ \emph {et~al.}(2017)\citenamefont {Yajima},
  \citenamefont {Phung-Duc},\ and\ \citenamefont {Masuyama}}]{Yajima:2017fv}%
  \BibitemOpen
  \bibfield  {author} {\bibinfo {author} {\bibfnamefont {M.}~\bibnamefont
  {Yajima}}, \bibinfo {author} {\bibfnamefont {T.}~\bibnamefont {Phung-Duc}},\
  and\ \bibinfo {author} {\bibfnamefont {H.}~\bibnamefont {Masuyama}},\
  }\href@noop {} {\bibfield  {journal} {\bibinfo  {journal} {ACM electric
  arvhive}\ } (\bibinfo {year} {2017})}\BibitemShut {NoStop}%
\bibitem [{\citenamefont {S.}\ \emph {et~al.}(2018)\citenamefont {S.},
  \citenamefont {Subbiah},\ and\ \citenamefont {Srinivasan}}]{S.:2018zc}%
  \BibitemOpen
  \bibfield  {author} {\bibinfo {author} {\bibfnamefont {H.}~\bibnamefont
  {S.}}, \bibinfo {author} {\bibfnamefont {M.}~\bibnamefont {Subbiah}},\ and\
  \bibinfo {author} {\bibfnamefont {M.~R.}\ \bibnamefont {Srinivasan}},\
  }\bibfield  {booktitle} {\emph {\bibinfo {booktitle} {Communications in
  Statistics: Case Studies, Data Analysis and Applications}},\ }\href
  {https://doi.org/10.1080/23737484.2018.1445979} {\bibfield  {journal}
  {\bibinfo  {journal} {Communications in Statistics: Case Studies, Data
  Analysis and Applications}\ }\textbf {\bibinfo {volume} {4}},\ \bibinfo
  {pages} {1} (\bibinfo {year} {2018})}\BibitemShut {NoStop}%
\bibitem [{Note1()}]{Note1}%
  \BibitemOpen
  \bibinfo {note} {The symbol $A/B/c$ is called the Kendall symbol of the
  queuing system, with the arrival process $A$, the service demand $B$, and the
  number of servers $c$. $G/G/\infty $ means that both the arrival process and
  the service demand are general and dependent, and have infinitely many
  servers. On the other hand, $M/M/\infty $ means that the arrival process is
  Poisson process and the service demand is independent exponential
  distribution.}\BibitemShut {Stop}%
\end{thebibliography}%

\end{document}